\documentclass[a4paper,10pt]{article}
\usepackage[english]{babel}
\usepackage[utf8]{inputenc}
\usepackage[T1]{fontenc}
\usepackage{textcomp} 
\usepackage{amsmath}
\usepackage{enumerate}
\usepackage{amsfonts}
\usepackage{amsthm}
\usepackage{mathrsfs}
\usepackage{amssymb}
\usepackage{comment}
\usepackage{csquotes}
\usepackage{bm}
\usepackage{graphicx}
\usepackage{bbm}
\usepackage{array}
\usepackage{ragged2e}
\usepackage{mathtools}
\usepackage{xcolor}
\usepackage{lipsum}
\usepackage{tabularx}
\usepackage[mathscr]{euscript}
\DeclareSymbolFont{rsfs}{U}{rsfs}{m}{n}
\DeclareSymbolFontAlphabet{\mathscrsfs}{rsfs}

\newcommand{\keywords}[1]{\par\noindent{\small\textbf{Keywords:} #1}}
\newcommand{\subjclass}[1]{\par\noindent{\small\textbf{AMS 2010 Subject Classification:} #1}}

\theoremstyle{definition}
\newtheorem{Def}{Definition}[section]

\theoremstyle{plain}
\newtheorem{Prop}[Def]{Proposition}
\newtheorem{thm}[Def]{Theorem}
\newtheorem{Thm}[Def]{Theorem}
\newtheorem{Lemma}[Def]{Lemma}

\newtheorem{rem}[Def]{Remark}

\newcommand{\R}{\mathbb{R}}

\renewcommand{\P}{\mathbb{P}}
\newcommand{\E}{\mathbb{E}}

\newcommand{\N}{\mathbb{N}}

\newcommand\HHH{\mathfrak{H}}
\newcommand\D{\mathbb{D}}

\newcommand\proj{\rm proj}

\DeclareMathOperator*{\Var}{Var}
\DeclareMathOperator*{\Cov}{Cov}

\newcolumntype{M}[1]{>{\raggedright}m{#1}}

\renewcommand{\epsilon}{\varepsilon}

\begin{document}

\title{Absolute continuity of finite-dimensional distributions of Hermite processes via Malliavin calculus}
\author{
  Laurent Loosveldt\thanks{Université de Liège, Département de Mathématiques, Liège, Belgium}%
  \and
  Yassine Nachit\thanks{Université de Lille, Laboratoire Paul Painlevé, Lille, France}%
  \and
    Ivan Nourdin\thanks{University of Luxembourg, Department of Mathematics, Esch-sur-Alzette, Luxembourg}%
  \and
  Ciprian Tudor\thanks{Université de Lille, Laboratoire Paul Painlevé, Lille, France}%
}
\date{}

\maketitle

\begin{abstract}
We investigate the existence of densities for finite-dimensional distributions of Hermite processes of order \(q \ge 1\) and self-similarity parameter \(H\in(\frac12,1)\). Whereas the Gaussian case \(q=1\) (fractional Brownian motion) is well understood, the non-Gaussian situation has not yet been settled. In this work, we extend the classical three-step approach used in the Gaussian case: factorization of the determinant into conditional terms, strong local nondeterminism, and non-degeneracy. We transport this strategy to the Hermite setting using Malliavin calculus. Specifically, we establish a determinant identity for the Malliavin matrix, prove strong local nondeterminism at the level of Malliavin derivatives, and apply the Bouleau-Hirsch criterion. Consequently, for any distinct times \(t_1,\dots,t_n\), the vector \((Z^{H,q}_{t_1},\dots,Z^{H,q}_{t_n})\) of a Hermite process admits a density with respect to the  Lebesgue measure. Beyond the result itself, the main contribution is the methodology, which could extend to other non-Gaussian models.
\end{abstract}



	\maketitle

\section{Introduction}

A cornerstone of Malliavin calculus is the \emph{Bouleau--Hirsch criterion} \cite{BouleauHirsch1984}, which asserts that a random vector 
$(Z_1,\dots,Z_n)$
with Malliavin differentiable components has an absolutely continuous distribution with respect to the Lebesgue measure on $\R^n$ as soon as its \emph{Malliavin matrix}
\[
\Big(\,\langle D Z_i,D Z_j\rangle\,\Big)_{1\le i,j\le n}
\]
is almost surely non-singular. This extends the Gaussian situation, where absolute continuity is equivalent to the non-degeneracy of the covariance matrix, which in that case coincides with the Malliavin matrix.

 \medskip

The case of vectors of multiple Wiener--Itô integrals has been recently investigated. For $n=1$, the situation is fully understood: a classical theorem of Shigekawa \cite{Shigekawa} ensures that any nontrivial multiple integral has a density with respect to Lebesgue measure. For higher-dimensional vectors, the problem becomes subtler. Nourdin, Nualart and Poly \cite{NNP} established a general dichotomy: if the components of a vector belong to a finite sum of Wiener chaoses, then either its law is absolutely continuous, or there exists a nontrivial polynomial relation annihilating the vector; in particular, the existence of a density can often be decided from structural properties of the chaos expansions. In a complementary direction, Nualart and Tudor \cite{NT} analyzed the case of a \emph{pair} $(I_q(f),I_q(g))$ of multiple integrals of the same order. They proved that such a pair fails to have a density if and only if the two integrals are proportional. They also noticed that this criterion does not hold in dimension bigger or equal than three. Their result provides a sharp criterion in dimension two, while the present paper focuses on vectors of  \emph{arbitrary} size with a specific stochastic process structure.

 \medskip
 
Our object of study is the class of \emph{Hermite processes} of order $q\ge1$ and self-similarity parameter $H\in(\frac12,1)$. They are non-Gaussian (except for $q=1$), self-similar, have stationary increments, live in the $q$-th Wiener chaos, and arise naturally as limits in the non-central limit theorems of Dobrushin--Major \cite{DM} and Taqqu \cite{Taqqu79}. For $q=1$, one recovers \emph{fractional Brownian motion}, the unique Gaussian self-similar process with stationary increments. The case $q=2$ corresponds to the \emph{Rosenblatt process}, first identified by Taqqu \cite{Taqqu75} in the study of quadratic functionals of long-range dependent Gaussian sequences, and also appearing in Rosenblatt’s work \cite{Rosenblatt} on non-Gaussian limit laws. The Rosenblatt process has since become the canonical example of a non-Gaussian, long-memory, self-similar process, often regarded as the natural analogue of fractional Brownian motion in the non-Gaussian world (see also Tudor \cite{Tudor08}). For higher orders $q\ge 3$, Hermite processes provide further canonical models of long-range dependence beyond Gaussianity.  

 \medskip

From a probabilistic perspective, Hermite processes occupy a central place in the limit theory of dependent time series. Whenever one considers normalized partial sums of a Gaussian sequence with long memory, transformed by a non-linear function of Hermite rank $q$, the limit in distribution is a Hermite process of order $q$. This universality explains their fundamental role in non-central limit theorems, time series analysis, and statistical inference for long-memory models. We refer to the surveys of Taqqu \cite{Taqqu03}, the monograph of Tudor \cite{Tudor23}, and the book by Pipiras and Taqqu \cite{PipirasTaqqu17} for a detailed account of their properties and applications.

\medskip

By the Bouleau--Hirsch criterion \cite{BouleauHirsch1984}, to prove that the finite-dimensional vector \((Z_1,\dots,Z_n)\) has a density it is enough to show that its Malliavin matrix is almost surely non-singular. In the case of Hermite processes, our strategy is to control the determinant of this matrix by combining two ingredients: a  factorization identity for the Malliavin matrix, and a strong local nondeterminism estimate formulated at the level of Malliavin derivatives. This leads to the following theorem.

\begin{thm}\label{thm1}
Let $n\ge 2$ and let $0<t_1<\cdots<t_n$. 
 Let $Z^{H,q}$ be a Hermite process of order $q\ge 1$ and self-similarity index $H\in(\frac12,1)$.
Then, the random vector
$
(Z^{H,q}_{t_1},\dots,Z^{H,q}_{t_n}),
$
admits a density with respect to Lebesgue measure on $\R^n$. 
\end{thm}

In order to motivate our approach for proving Theorem \ref{thm1}, let us first recall the classical Gaussian case $q=1$ (fractional Brownian motion). 
In that case, a simple yet elegant three-step argument establishes that the finite-dimensional distributions possess a density; in other words, Theorem \ref{thm1} holds when $q=1$.
This classical scheme will serve as a blueprint for our generalization to the Hermite setting.

\medskip
\noindent\textbf{Sketch of the proof of Theorem \ref{thm1} when $q=1$.}
Let $B^H=(B^H_t)_{t\ge0}$ denote a fractional Brownian motion with Hurst parameter $H\in(0,1)$. 
Fix $0<t_1<\cdots<t_n$ and set $Z_k:=B^H_{t_k}$.  

\medskip

\emph{Step 1 (determinant and conditional variances).}  
For any Gaussian vector, the determinant of its covariance matrix is equal to the product of conditional variances (see, e.g., \cite[Eq. (2.8) p. 71]{Berman}):
\begin{equation}\label{det}
\det \Cov (Z_1,\dots,Z_n)\;=\;\Var(Z_1)\prod_{k=2}^n \Var(Z_k\mid Z_1,\dots,Z_{k-1}).
\end{equation}

\medskip

\emph{Step 2 (SLND).} 
By using the self-similarity and the stationarity of the increments, Pitt \cite[Lemma 7.1]{Pitt} showed that the fractional Brownian motion enjoys the \emph{strong local nondeterminism} property: there exists $c_H>0$ such that
\begin{equation}\label{eq:SLND-general}
\Var\!\Big(Z_{k}\,\Big|\,Z_{1},\dots,Z_{{k-1}}\Big) \;\ge\; c_H\,(t_k-t_{k-1})^{2H}, \qquad k=2,\dots,n.
\end{equation}

\medskip

\emph{Step 3 (conclusion).}  
Since $(Z_1,\dots,Z_n)$ is Gaussian, non-degeneracy of its covariance matrix implies absolute continuity. 
Combining Steps~1 and 2, we obtain
\[
\det \Cov(Z_1,\dots,Z_n)\;\ge\;c_H^n\prod_{k=1}^n (t_k-t_{k-1})^{2H}\;>\;0,
\]
so the law of $(Z_{1},\dots,Z_{n})$ has a density.

\medskip
 
In the non-Gaussian case $q \ge 2$, we adopt the same guiding philosophy, but the passage from the Gaussian to the non-Gaussian setting requires substantial innovations: the covariance matrix is replaced by the Malliavin matrix, and the determinant identity \eqref{det} for the covariance matrix by the corresponding identity for the Malliavin matrix, see Lemma \ref{Lemma:detGamma} below. Likewise, SLND is replaced by a new analogue at the level of Malliavin derivatives. To establish this extension of SLND, we prove a self-similarity and stationarity of increments property for the Malliavin derivative.
 
 \medskip
 
This new framework not only yields absolute continuity of Hermite finite-dimensional distributions but, more importantly, it provides a robust methodological advance: we design a systematic way to transport the classical Gaussian scheme into the non-Gaussian Malliavin setting. Beyond the scope of Hermite processes, we believe this methodology opens the door to treating absolute continuity questions for a wide range of non-Gaussian processes representable as multiple Wiener–Itô integrals.
 
\medskip
The rest of the paper is organized as follows.
Section~\ref{sec:Malliavin} recalls the necessary background on Malliavin calculus.  
In Section~\ref{sec:Hermite} we present the definition and main properties of Hermite processes.  
Section~\ref{sec:det} contains a determinant identity for the Malliavin matrix. 
%Namely, as this matrix is a Gram matrix, this identity is simply the classical factorization formula of the determinant of a Gram matrix into a product of squared distances, expressed within the context of Malliavin calculus.
In Section~\ref{sec:malliavinselfsimilarity} we establish self-similarity and stationarity of increments for  the Malliavin derivative.  
Section~\ref{sec:SLND} develops the analogue of the strong local nondeterminism property in this framework.  
Finally, Section~\ref{sec:proof} combines these ingredients with the Bouleau--Hirsch criterion to prove Theorem~\ref{thm1}.

\section{Preliminaries on Malliavin calculus}\label{sec:Malliavin}

This section recalls the basic tools of Malliavin calculus that will be used throughout the paper.  
We refer to the standard references \cite{nourdin2012normal, Nualart, Sanz-Sole} for a complete account.

\medskip
\noindent\textbf{Isonormal Gaussian process.}  
Let $B=\{B(t),\,t\in\R\}$ be a two-sided  Brownian motion on a complete probability space $(\Omega,\mathcal F,\P)$, equipped with its natural filtration and such that $B(0)=0$.  
We denote by $\HHH=L^2(\R)$ the canonical Hilbert space associated with $B$. For each $h\in\HHH$, the Wiener integral
\[
B(h)=\int_\R h(t)\,dB(t)
\]
is defined first for step functions and then extended by density.  
The family $\{B(h):h\in\HHH\}$ forms an \emph{isonormal Gaussian process} over $\HHH$, that is, a centred Gaussian family satisfying
\[
\E[B(h)B(g)] = \langle h,g\rangle_{\HHH}, \qquad h,g\in\HHH.
\]

\medskip
\noindent\textbf{Hermite polynomials and Wiener chaoses.}  
Let $(H_q)_{q\ge0}$ be the Hermite polynomials, defined by the recursion
\[
H_0(x)=1,\quad H_1(x)=x,\quad H_{q+1}(x)=xH_q(x)-qH_{q-1}(x).
\]
For $q\ge1$, the \emph{$q$th Wiener chaos} $\mathcal H_q$ is defined as the closed linear span in $L^2(\Omega)$ of the random variables
\[
\{H_q(B(h)):\, h\in\HHH,\ \|h\|_{\HHH}=1\}.
\]
In particular, $\mathcal H_1$ is the Gaussian chaos generated by $B(h)$, while $\mathcal H_q$ is non-Gaussian as soon as $q\ge2$.  
The family $(\mathcal H_q)_{q\ge0}$ forms an orthogonal decomposition
\[
L^2(\Omega)=\bigoplus_{q=0}^\infty \mathcal H_q,
\]
known as the \emph{Wiener–Itô chaos expansion}.

\medskip
\noindent\textbf{Multiple Wiener–Itô integrals.}  
For $q\ge1$ and $h\in\HHH$ with $\|h\|_{\HHH}=1$, we define
\[
I_q(h^{\otimes q}) := H_q(B(h)).
\]
If $\|h\|_{\HHH}\ne1$, one sets $I_q(h^{\otimes q}) := \|h\|_{\HHH}^q H_q(B(h/\|h\|_{\HHH}))$.  
For a general elementary tensor $h_1\otimes\cdots\otimes h_q\in\HHH^{\otimes q}$ (the symmetric tensor product), one defines $I_q(h_1\otimes\cdots\otimes h_q)$ by polarization, and for a general symmetric function $f\in\HHH^{\odot q}$, one extends linearly and by density.  
This yields a well-defined linear operator
\[
I_q:\HHH^{\odot q}\to L^2(\Omega).
\]
satisfying the isometry
\[
\E[I_p(f)I_q(g)] =
\begin{cases}
p!\,\langle f,g\rangle_{\HHH^{\otimes p}}, & p=q,\\
0, & p\ne q,
\end{cases}
\]
for all $f\in\HHH^{\odot p}$, $g\in\HHH^{\odot q}$.  

One has the fundamental equivalence
$$
\mathcal H_q = \{ I_q(f): f\in \HHH^{\odot q}\}, \qquad q\ge1,
$$
which shows that the two constructions of Wiener chaos (via Hermite polynomials  or via multiple integrals) are identical.

\medskip

\noindent\textbf{Product formula for multiple integrals.}  
A key tool is the product formula: for $f\in\HHH^{\odot p}$ and $g\in\HHH^{\odot q}$,
\begin{equation}\label{prod}
I_p(f)I_q(g)= \sum_{r=0}^{p\wedge q} r!\binom{p}{r}\binom{q}{r}\,
I_{p+q-2r}\big(f\widetilde{\otimes}_r g\big),
\end{equation}
where $f\otimes_r g$ denotes the $r$-contraction of $f$ and $g$, defined by
\begin{align}\label{contra}
(f\otimes_r g)(t_{1},\ldots,t_{p+q-2r})
&= \int_{\R^{r}} f(u_{1},\ldots,u_{r}, t_{1},\ldots,t_{p-r}) \nonumber \\
&\quad \times g(u_{1},\ldots,u_{r}, t_{p-r+1},\ldots,t_{p+q-2r})\,du_{1}\cdots du_{r},
\end{align}
for $1\le r\le p\wedge q$, and $f\otimes_0 g = f\otimes g$ is the tensor product.  
In general, $f\otimes_r g\in\HHH^{\otimes(p+q-2r)}$ is not symmetric, and we denote by $f\widetilde{\otimes}_r g$ its symmetrization.  
For more details, see e.g. Section 1.1.2 in \cite{Nualart}.

\medskip
\noindent\textbf{Malliavin derivative.}  
Let $\mathcal S$ be the class of \emph{smooth cylindrical random variables} of the form
\[
F = f(B(g_1),\dots,B(g_n)),
\]
with $n\ge1$, $g_1,\dots,g_n\in\HHH$, and $f\in C^\infty_P(\R^n)$ (i.e., smooth with all derivatives of polynomial growth).  
The \emph{Malliavin derivative} of $F$ is the stochastic process
\[
D_r F = \sum_{j=1}^n \frac{\partial f}{\partial x_j}(B(g_1),\dots,B(g_n))\, g_j(r), \qquad r\in\R.
\]
By iteration, one defines higher-order derivatives $D^kF$.  
For $k,p\ge1$, the Sobolev-type space $\D^{k,p}$ is the closure of $\mathcal S$ with respect to the norm
\[
\|F\|_{k,p}^p = \E[|F|^p] + \sum_{j=1}^k \E\!\left[\|D^jF\|_{\HHH^{\otimes j}}^p\right],
\]
and we set $\D^\infty = \bigcap_{k,p\ge1}\D^{k,p}$.  
When $F,G\in \D^{1,2}$, the inner product of their derivatives reads
\[
\langle DF,DG\rangle_\HHH = \int_\R D_rF\,D_rG\,dr.
\]

\medskip
\noindent\textbf{Bouleau--Hirsch criterion.}  
A central application of Malliavin calculus is the analysis of regularity of distributions.  
As already mentioned, the following classical result will play a key role.

\begin{Thm}[Bouleau--Hirsch \cite{BouleauHirsch1984}]\label{Thm:2.1.2}
Let ${\bf Z}=(Z_1,\dots,Z_n)$ be a random vector such that each $Z_i\in\D^{1,p}$ for some $p>1$.  
If its Malliavin matrix
\[
\Gamma_{\bf Z} = \big(\langle DZ_i,DZ_j\rangle_\HHH\big)_{1\le i,j\le n}
\]
is a.s. non-degenerate, i.e. $\det(\Gamma_{\bf Z})>0$ a.s., then the law of ${\bf Z}$ is absolutely continuous with respect to the Lebesgue measure on $\R^n$.
\end{Thm}
%
%\medskip
%\noindent\textbf{Absolute continuity in fixed chaos.}  
%Finally, the following lemma will be useful.

%\begin{Lemma}[Shigekawa \cite{Shigekawa}]\label{Lemma:det-pos}
%Let $Z$ be a nonzero random variable belonging to the $q$th Wiener chaos with $q\ge1$.  
%Then
%\[
%\P\big(\|DZ\|_\HHH>0\big)=1.
%\]
%In particular, $Z$ admits a density with respect to the Lebesgue measure on $\R$.
%\end{Lemma}

\section{Background on Hermite processes}\label{sec:Hermite}

We now introduce Hermite processes and recall some of their fundamental properties.  
For a detailed account, we refer to Tudor’s monograph \cite{Tudor23}.

\begin{Def}\label{Def:Hermite}
Let $q\ge1$ and $H\in(\frac12,1)$.  
The \emph{Hermite process} $(Z_t^{H,q})_{t\in\R}$ of order $q$ and self-similarity parameter $H$ is defined by
\begin{equation}\label{eq:Zt}
Z_t^{H,q} \;=\; I_q(L_t^{H,q}), \qquad 
L_t^{H,q} = c(H,q)\int_0^t \prod_{j=1}^q (s-\xi_j)_{+}^{H_0-\tfrac32}\,ds, \hskip0.2cm t\in \R,
\end{equation}
where $H_0 = 1+\tfrac{H-1}{q}\in(1-\tfrac1{2q},1)$ and
\[
c(H,q) = \sqrt{\frac{H(2H-1)}{q!\,\beta^q(H_0-\tfrac12,\,2-2H_0)}}.
\]
Here $\beta(\cdot,\cdot)$ denotes the Beta function, and $I_q$ is the multiple Wiener--Itô integral of order $q$.
\end{Def}

\begin{rem}
{\rm
\begin{itemize}
  \item The constant $c(H,q)$ has been chosen so that $\E[(Z_1^{H,q})^2]=1$.
  \item We use the convention $\theta^\alpha_+ = \theta^\alpha$ if $\theta>0$, and $0$ otherwise.  
        Integrals on negative intervals are defined by $\int_0^t  = -\int_t^0 $ when $t<0$.
  \item In the literature, Hermite processes are usually defined only for $t\ge0$.  
        In analogy with the moving-average representation of fractional Brownian motion (see, e.g., \cite{hult2003topics}, pp.~9--10), we have extended their definition here to the whole real line.
        \item For each fixed $t\in\R$, the random variable $Z_t^{H,q}$ is a multiple Wiener--Itô integral.  
\item The class of Hermite processes contains important examples:
\begin{itemize}
  \item for $q=1$, one recovers fractional Brownian motion, 
  \item for $q=2$, one obtains the Rosenblatt process,
  \item for $q\ge3$, one obtains higher-order non-Gaussian processes.
\end{itemize}
Fractional Brownian motion is the only Gaussian Hermite process; all others are non-Gaussian.
\end{itemize}
}
\end{rem}

\medskip
Hermite processes enjoy the following key properties (see \cite{Tudor23} for proofs in the case $t\ge0$; the extension to $t\in\R$ being immediate):

\begin{itemize}
  \item \textit{Self-similarity.} For every $a>0$,
  \[
  (Z_{at}^{H,q})_{t\in\R} \;\stackrel{\text{law}}{=}\; (a^H Z_t^{H,q})_{t\in\R}.
  \]
  \item \textit{Stationary increments.} For every $h\in\R$,
  \[
  (Z_{t+h}^{H,q}-Z_h^{H,q})_{t\in\R} \;\stackrel{\text{law}}{=}\; (Z_t^{H,q})_{t\in\R}.
  \]
  \item \textit{Covariance function.} For all $s,t\in\R$,
  \[
  \E[Z_t^{H,q}Z_s^{H,q}] \;=\; \tfrac12\big(|t|^{2H}+|s|^{2H}-|t-s|^{2H}\big),
  \]
  which coincides with the covariance of fractional Brownian motion with Hurst index $H$.
  \item \textit{Adaptedness.} The process $(Z_t^{H,q})_{t\in\R}$ is adapted to the Brownian filtration, and
  \begin{equation}\label{eq:adapt}
  D_r Z_t^{H,q} = 0 \qquad \text{for all } r>t.
  \end{equation}
\end{itemize}

\begin{rem}
From stationarity of increments, one readily checks that $Z_t^{H,q}\overset{\text{law}}{=}-Z_{-t}^{H,q}$ for every $t\in\R$.
\end{rem}

Finally, the following lemma will be useful.
\begin{Lemma}\label{lem:positive}
  Let  $Z^{H,q}=\left(Z^{H,q}_t\right)_{t\in \R}$ be the Hermite process of order $q \geqslant 1$ and self-similarity parameter $H \in\left(\frac{1}{2}, 1\right)$. We have
  \begin{align}
    \|DZ^{H,q}_1\|_{\HHH_1}^2 >0 \qquad \mathbb{P}\text{-a.s.}
  \end{align}
  where $\HHH_1= L^2([0,1])$ and $\|f\|_{\HHH_{1}}^2=\int_0^1|f(x)|^2 dx$. 
\end{Lemma}
\begin{proof}
  The proof follows the same lines as the proof of $(a)\Rightarrow(c)$ in  \cite[Theorem 3.1]{NNP}, plus the fact that 
  \begin{align*}
    &\E\left[\|DZ^{H,q}_1\|_{\HHH_1}^2\right] \\
    &= c(H, q)^2\,q^2(q-1)!\,\int_0^1dx_1\int_{\R^{q-1}}\prod_{l=2}^{q}dx_l\left\{\int_0^{1}ds\prod_{k=1}^{q}\left(s-x_{k}\right)_{+}^{H_0-\frac{3}{2}}\right\}^2>0.
  \end{align*}
\end{proof}

\section{A determinant identity for the Malliavin matrix}\label{sec:det}

We now state our determinant identity for the Malliavin matrix, which should be seen as the analogue of identity~\eqref{det} in the Gaussian case for the covariance matrix.  
Our proof is nothing but an adaptation of the classical argument used to establish the determinant of a Gram matrix.

\begin{Lemma}\label{Lemma:detGamma}
Let ${\bf Z}=(Z_1,\ldots,Z_n) \in (\mathbb{D}^{1,1})^n$, and let $\Gamma_{\bf Z}$ denote its Malliavin matrix. Then, almost surely,
\begin{align}\label{eq:detGamma88}
\det(\Gamma_{\bf Z})
= \|DZ_1\|_{\HHH}^2
  \prod_{j=2}^n \big\|DZ_j - {\rm proj}_{E_{j-1}}(DZ_j)\big\|_{\HHH}^2,
\end{align}
where $E_{j-1}$ is the linear span in $\HHH$ of $\{DZ_1,\ldots,DZ_{j-1}\}$, and ${\rm proj}_{E_{j-1}}$ denotes the orthogonal projection onto $E_{j-1}$.
\end{Lemma}

\begin{proof}
Recall that the Malliavin matrix of ${\bf Z}$ has entries
\[
\Gamma_{i,j}=\langle DZ_i, DZ_j\rangle_{\HHH}, \qquad 1\le i,j\le n.
\]

Fix $i=n$. We decompose
\begin{eqnarray*}
\langle DZ_j, DZ_n\rangle_{\HHH}
&=& \langle DZ_j, {\rm proj}_{E_{n-1}}(DZ_n)\rangle_{\HHH}
  + \langle DZ_j, DZ_n-{\rm proj}_{E_{n-1}}(DZ_n)\rangle_{\HHH}\\
&=:& A_{j,n}+B_{j,n}.
\end{eqnarray*}
Since $DZ_n-{\rm proj}_{E_{n-1}}(DZ_n)$ is orthogonal to $E_{n-1}$, we immediately obtain
\[
B_{j,n}=0,\qquad j=1,\ldots,n-1.
\]

By multilinearity of the determinant, we may split
\[
\det(\Gamma_{\bf Z})=\det(A_{\bf Z})+\det(B_{\bf Z}),
\]
where $\det(A_{\bf Z})$ is obtained by replacing the last column of $\Gamma_{\bf Z}$ with $(A_{1,n},\ldots,A_{n,n})^\top$, and $\det(B_{\bf Z})$ with $(B_{1,n},\ldots,B_{n,n})^\top$.
But by definition of the projection, the last column of $A_{\bf Z}$ is a linear combination of the first $n-1$, hence $\det(A_{\bf Z})=0$.  

On the other hand, $\det(B_{\bf Z})$ has the block form
\begin{eqnarray*}
\det(B_{\bf Z})&=&
\det\!\begin{pmatrix}
\big(\langle DZ_i,DZ_j\rangle_{\HHH}\big)_{1\le i,j\le n-1} & 0 \\
\ast & B_{n,n}
\end{pmatrix}\\
&=& B_{n,n}\,\det\big(\langle DZ_i,DZ_j\rangle_{\HHH}\big)_{1\le i,j\le n-1}.
\end{eqnarray*}
Since $B_{n,n}=\|DZ_n-{\rm proj}_{E_{n-1}}(DZ_n)\|_{\HHH}^2$, this yields
\[
\det(\Gamma_{\bf Z})=\|DZ_n-{\rm proj}_{E_{n-1}}(DZ_n)\|_{\HHH}^2
\det\big(\langle DZ_i,DZ_j\rangle_{\HHH}\big)_{1\le i,j\le n-1}.
\]

The desired factorization \eqref{eq:detGamma88} then follows by induction on $n$.
\end{proof}

\section{Malliavin self-similarity and Malliavin \\stationarity of increments of Hermite processes} \label{sec:malliavinselfsimilarity}

To establish the analogue of the strong local nondeterminism estimate~\eqref{eq:SLND-general} for Hermite processes, we first need to strengthen the classical properties of self-similarity and stationarity of increments, now at the level of their Malliavin derivative.  
This is the content of the following proposition.

\begin{Prop}\label{Prop:Zt9999}
Let $Z^{H,q}=\{Z^{H,q}_t,\,t\in\R\}$ be the Hermite process of order $q\ge1$ and self-similarity parameter $H\in(\tfrac12,1)$  given by (\ref{eq:Zt}).  
Then the Malliavin derivative of $Z^{H,q}$ inherits both stationarity of increments and self-similarity: for all $a\in\R$ and $c>0$,
\[
\Big(\,\langle DZ^{H,q}_{s+a}-DZ^{H,q}_{a},\, DZ^{H,q}_{t+a}-DZ^{H,q}_{a}\rangle_{\HHH}\,\Big)_{s,t\in\R}
\;\overset{\rm law}{=}\;
\Big(\,\langle DZ^{H,q}_{s},\, DZ^{H,q}_{t}\rangle_{\HHH}\,\Big)_{s,t\in\R},
\]
and
\[
\Big(\,\langle DZ^{H,q}_{cs},\, DZ^{H,q}_{ct}\rangle_{\HHH}\,\Big)_{s,t\in\R}
\;\overset{\rm law}{=}\; c^{2H}\,
\Big(\,\langle DZ^{H,q}_{s},\, DZ^{H,q}_{t}\rangle_{\HHH}\,\Big)_{s,t\in\R}.
\]
\end{Prop}

\begin{proof}
Let us start by establishing that the Malliavin derivative of $Z^{H,q}$ is self-similar. Let $c>0$. According to the product formula (\ref{prod}) and the definition of the contraction (\ref{contra}), we have, for all $s,t \in \R$,
\begin{align}\label{eq:ZZcsct}
	\langle DZ^{H,q}_{cs},\, DZ^{H,q}_{ct}\rangle_{\HHH} &= q^{2} \int_{\R} dx_q\,  I_{q-1}\big(L _{cs}^{H,q}(\star, x_q)\big)I_{q-1}\big(L _{ct}^{H,q}(\star, x_q)\big)\nonumber\\
	&= q^{2} \sum_{r=0}^{q-1} r! \binom{q-1}{r}^2 I _{2(q-1)-2r}\left( L^{H,q}_{cs}\widetilde{\otimes} _{r+1} L_{ct} ^{H,q}\right),
\end{align}
with $L^{H,q}$ is given by (\ref{eq:Zt}). For $t_1,t_2\in \R$, $r=0,\ldots,q-1$, and   $y^{1},y^2\in \R^{q-1-r}$,
\begin{align}\label{eq:contracLt1t2}
	&(L^{H,q}_{t_1}\otimes _{r+1}L ^{H,q}_{t_2})(y^{1}, y^2)\nonumber\\
	&= c(H,q) ^{2} \int_{\R ^{r+1}}\prod_{i=1}^{r+1}dx_{i} \prod_{j=1}^{2}\int_{0} ^{t_j} du_j  \prod_{k=1}^{r+1}(u_j-x_{k})_{+} ^{H_0-\tfrac32}\prod_{\ell=1}^{q-1-r}(u_j-y^j_{\ell})_{+} ^{H_0-\tfrac32}\nonumber\\
	&=a(H,q,r) \int_{0}^{t_1}du_1 \int_{0} ^{t_2} du_2 \vert u_1-u_2\vert ^{\frac{2(H-1)(r+1)}{q}}\prod_{j=1}^{2}\prod_{\ell=1}^{q-1-r}(u_j-y^j_{\ell})_{+} ^{H_0-\tfrac32},
\end{align}
where $a(H,q,r):= c(H,q) ^{2} \beta \left( \frac{2-2H}{q}, \frac{1}{2}-\frac{1-H}{q}\right)^{r+1}$, $\beta$ is the beta function,  and in the last equality we used the following formula: for $-1<a<-\frac{1}{2}$,
$$\int_{\R}(u-y)_{+} ^{a} (v-y)_{+} ^{a} du= \beta(-1-2a, a+1)\vert u-v\vert ^{2a+1}.$$
For $c>0$, $s,t\in \R$, combining \eqref{eq:ZZcsct} and \eqref{eq:contracLt1t2} yields 
\begin{align*}
	&\langle DZ^{H,q}_{cs},\, DZ^{H,q}_{ct}\rangle_{\HHH} =q^{2} \sum_{r=0}^{q-1} a(H,q,r) r! \binom{q-1}{r}^2 \int_{\R ^{2(q-1)-2r}} \prod_{m=1}^{2}\prod_{k=1}^{q-1-r}dB(y^m_{k})\\ 
	&\qquad\qquad\qquad\quad \times \int_{0}^{cs}du_1 \int_{0} ^{ct} du_2 \vert u_1-u_2\vert ^{\frac{2(H-1)(r+1)}{q}}\prod_{j=1}^{2}\prod_{\ell=1}^{q-1-r}(u_j-y^j_{\ell})_{+} ^{H_0-\tfrac32}.
\end{align*}
By the change of variables $v_j=\frac{u_j}{c}$ and $z^j_{\ell}=\frac{y^j_{\ell}}{c}$, for $j=1,2$,  and $\ell=1,..., q-1-r$,   we obtain
\begin{align*}
	&\langle DZ^{H,q}_{cs},\, DZ^{H,q}_{ct}\rangle_{\HHH}\\
	&=c^{2H-(q-1-r)}q^{2} \sum_{r=0} ^{q-1} a(H,q,r) r!  \binom{q-1}{r}^2 \int_{\R ^{2(q-1)-2r}} \prod_{m=1}^{2}\prod_{k=1}^{q-1-r}dB(cz^m_{k})\\
	&\qquad\times \int_{0}^{s}dv_1 \int_{0} ^{t} dv_2 \vert v_1-v_2\vert ^{\frac{2(H-1)(r+1)}{q}}\prod_{j=1}^{2}\prod_{\ell=1}^{q-(r+1)}(v_j-z^j_{\ell})_{+} ^{H_0-\tfrac32},
\end{align*}
using the scaling property of the Wiener process. Let us observe that the finite-dimensional distributions of the two-parameter process on the right-hand side of the the previous equality coincide in law with those of the two-parameter process given by:
\begin{align*}
	&c^{2H-(q-1-r)} c^{q-r-1} q^{2} \sum_{r=0} ^{q-1} a(H,q,r) r!  \binom{q-1}{r}^2 \int_{\R ^{2(q-1)-2r}} \prod_{m=1}^{2}\prod_{k=1}^{q-1-r}dB(z^m_{k}) \\
	&\quad\times\int_{0}^{s}dv_1 \int_{0} ^{t} dv_2 \vert v_1-v_2\vert ^{\frac{2(H-1)(r+1)}{q}}\prod_{j=1}^{2}\prod_{\ell=1}^{q-(r+1)}(v_j-z^j_{\ell})_{+} ^{H_0-\tfrac32}\\
	&=c^{2H}\langle DZ^{H,q}_{s},\, DZ^{H,q}_{t}\rangle_{\HHH}.
\end{align*} 
\indent We now prove the stationarity of increments property for the Malliavin derivative of $Z^{H,q}$. We have
\begin{align*}
	&\langle D(Z^{H,q}_{s+a}-Z^{H,q}_{a}),D(Z^{H,q}_{t+a}-Z^{H,q}_{a}) \rangle _{\HHH}\\
	&=  q^{2}  \sum_{r=0}^{q-1} r! \binom{q-1}{r}^2 I _{2(q-1)-2r}\left( (L^{H,q}_{s+a}-L ^{H,q}_{a})\widetilde{\otimes}_{r+1} (L_{t+a} ^{H,q}- L ^{H,q}_{a})\right)\\
	&=q^{2}\sum_{r=0} ^{q-1} r! a(H,q,r) \binom{q-1}{r}^2  \int_{\R ^{2(q-1)-2r}} \prod_{m=1}^{2}\prod_{k=1}^{q-1-r}dB(y^m_{k})\\
	&\times \int_{a}^{s+a}du_1 \int_{a} ^{t+a} du_2 \vert u_1-u_2\vert ^{\frac{2(H-1)(r+1)}{q}}\prod_{j=1}^{2}\prod_{\ell=1}^{q-1-r}(u_j-y^j_{\ell})_{+} ^{H_0-\tfrac32}.
\end{align*}
Using the change of variable $v_j=u_j-a$, for $j=1,2$,  and the stationarity of increments for the Wiener process $B$, we deduce: 
\[
\Big(\,\langle DZ^{H,q}_{s+a}-DZ^{H,q}_{a},\, DZ^{H,q}_{t+a}-DZ^{H,q}_{a}\rangle_{\HHH}\,\Big)_{s,t\in\R}
\;\overset{\rm law}{=}\;
\Big(\,\langle DZ^{H,q}_{s},\, DZ^{H,q}_{t}\rangle_{\HHH}\,\Big)_{s,t\in\R},
\]
which completes the proof of Proposition \ref{Prop:Zt9999}.
\end{proof}

\section{Strong local nondeterminism for Hermite processes}\label{sec:SLND}

In this section we establish the analogue of the strong local nondeterminism property for Hermite processes.  
The following theorem provides its precise formulation in the Malliavin framework.

\begin{Thm}\label{Thm:MLNDDiscrete}
Let $Z^{H,q}=\{Z^{H,q}_t,\,t\in\R\}$ be the Hermite process of order $q\ge1$ with self-similarity parameter $H\in(\tfrac12,1)$.  
For any $j\in\N^*$ and any time grid $0=t_0<t_1<\cdots<t_{j-1}<t_j$, one has
\begin{align}\label{9i-1}
\big\| DZ^{H,q}_{t_j} - {\rm proj}_{E_{j-1}}\big(DZ^{H,q}_{t_j}\big) \big\|_{\HHH}^2
\;\overset{\text{law}}{\ge}\; (t_j-t_{j-1})^{2H}\,\|DZ^{H,q}_1\|_{\HHH_{1}}^2,
\end{align}
where $\HHH_1=L^2([0,1])$, $E_{j-1}=\mathrm{span}\{DZ^{H,q}_{t_1},\ldots,DZ^{H,q}_{t_{j-1}}\}$, 
and the notation $X\overset{\text{law}}{\ge}Y$ means that
\[
\mathbb{E}[f(X)] \,\geq\, \mathbb{E}[f(Y)]
\]
for all non-decreasing bounded functions $f:\R_+\to\R$.
\end{Thm}

\begin{proof}[Proof of Theorem \ref{Thm:MLNDDiscrete}]
Fix $j\in \N^*$ and a time grid $0=t_0<t_1<\cdots<t_{j-1}<t_j$.  
Consider the random vector space
\[
\widetilde{E}_{j-1}=\operatorname{span}\big\{D(Z^{H,q}_{t_k}-Z^{H,q}_{t_\ell}) : k,\ell=0,\ldots,j-1\big\}.
\] 	
Let $(a_1,\ldots,a_{j-1})$ be the random variables such that 
\[
\big\| D Z^{H,q}_{t_j} - {\rm proj}_{E_{j-1}}(D Z^{H,q}_{t_j}) \big\|_{\HHH}^2
= \big\| D Z^{H,q}_{t_j} - \sum_{k=1}^{j-1} a_k D Z^{H,q}_{t_k} \big\|_{\HHH}^2.
\]
We claim that, almost surely
\begin{eqnarray}
&&\big\| DZ^{H,q}_{t_j} - {\rm proj}_{E_{j-1}}(DZ^{H,q}_{t_j}) \big\|_{\HHH}^2 \notag\\
&\ge& \big\| D(Z^{H,q}_{t_j}-Z^{H,q}_{t_{j-1}}) - {\rm proj}_{\widetilde{E}_{j-1}}\big(D(Z^{H,q}_{t_j}-Z^{H,q}_{t_{j-1}})) \big\|_{\HHH}^2.\label{9i-2}
\end{eqnarray}
Indeed, one can rewrite
\[
\left\| DZ^{H,q}_{t_j} - \sum_{k=1}^{j-1} a_k DZ^{H,q}_{t_k} \right\|_{\HHH}^2 
= \Big\| D(Z^{H,q}_{t_j}-Z^{H,q}_{t_{j-1}}) - \sum_{k,\ell=0}^{j-1} b_{k,\ell} D(Z^{H,q}_{t_k}-Z^{H,q}_{t_\ell}) \Big\|_{\HHH}^2,
\]
with coefficients $b_{k,0}=a_k$ for $k=1,\ldots,j-2$, $b_{j-1,0}=a_{j-1}-1$, and $b_{k,\ell}=0$ otherwise.  
Taking the infimum over $(b_{k,\ell})$ yields \eqref{9i-2}.  

Now, for any $(b_{k,\ell})$,
\begin{align*}
&\Big\| D(Z^{H,q}_{t_j}-Z^{H,q}_{t_{j-1}}) - \sum_{k,\ell=0}^{j-1} b_{k,\ell} D(Z^{H,q}_{t_k}-Z^{H,q}_{t_\ell}) \Big\|_{\HHH}^2 \\
&= \Big\| D\!\big(Z^{H,q}_{t_{j-1}+(t_j-t_{j-1})}-Z^{H,q}_{t_{j-1}}\big) 
     - \sum_{k,\ell=0}^{j-1} b_{k,\ell}\,D\!\big(Z^{H,q}_{t_{j-1}+(t_k-t_{j-1})}-Z^{H,q}_{t_{j-1}}\big) \\
&\hspace{5cm} + \sum_{k,\ell=0}^{j-1} b_{k,\ell}\,D\!\big(Z^{H,q}_{t_{j-1}+(t_\ell-t_{j-1})}-Z^{H,q}_{t_{j-1}}\big)\Big\|_{\HHH}^2\\
&=:X(b_{k,\ell}).
\end{align*}
By Malliavin stationarity of increments (Proposition \ref{Prop:Zt9999}), $X$ has the same law as
\[
\widetilde{X}(b_{k,\ell})=\Big\| DZ^{H,q}_{t_j-t_{j-1}} - \sum_{k,\ell=0}^{j-1} b_{k,\ell}\,D(Z^{H,q}_{t_k-t_{j-1}}-Z^{H,q}_{t_\ell-t_{j-1}})\Big\|_{\HHH}^2.
\]
By Malliavin self-similarity  (Proposition \ref{Prop:Zt9999}), the law of $\widetilde{X}$ is the same as that of
\[
\hat{X}(b_{k,\ell})=(t_j-t_{j-1})^{2H}\,\Big\| DZ^{H,q}_{1} - \sum_{k,\ell=0}^{j-1} b_{k,\ell}\,D\!\Big(Z^{H,q}_{\frac{t_k-t_{j-1}}{t_j-t_{j-1}}}-Z^{H,q}_{\frac{t_\ell-t_{j-1}}{t_j-t_{j-1}}}\Big)\Big\|_{\HHH}^2.
\]

Therefore,
\begin{align*}
&\big\| D(Z^{H,q}_{t_j}-Z^{H,q}_{t_{j-1}}) - {\rm proj}_{\widetilde{E}_{j-1}}(D(Z^{H,q}_{t_j}-Z^{H,q}_{t_{j-1}})) \big\|_{\HHH}^2\\
&\;\overset{\text{law}}{=}\; (t_j-t_{j-1})^{2H}\,
   \big\| DZ^{H,q}_{1} - {\rm proj}_{\hat{E}_{j-1}}(DZ^{H,q}_{1}) \big\|_{\HHH}^2,
\end{align*}
where
\[
\hat{E}_{j-1}=\operatorname{span}\Big\{ D\!\Big(Z^{H,q}_{\frac{t_k-t_{j-1}}{t_j-t_{j-1}}}-Z^{H,q}_{\frac{t_\ell-t_{j-1}}{t_j-t_{j-1}}}\Big): k,\ell=0,\ldots,j-1\Big\}.
\]

Finally, for any $(b_{k,\ell})$,
\[
\Big\| DZ^{H,q}_{1} - \sum_{k,\ell=0}^{j-1} b_{k,\ell}\,D\!\Big(Z^{H,q}_{\frac{t_k-t_{j-1}}{t_j-t_{j-1}}}-Z^{H,q}_{\frac{t_\ell-t_{j-1}}{t_j-t_{j-1}}}\Big)\Big\|_{\HHH}^2
\;\ge\; \|DZ^{H,q}_1\|_{\HHH_1}^2,
\]
since $D_r Z^{H,q}_{\frac{t_k-t_{j-1}}{t_j-t_{j-1}}}=0$ for $r\in[0,1]$ and $k=1,\dots,j-1$ by \eqref{eq:adapt}.  
Thus,
\[
\big\| D(Z^{H,q}_{t_j}-Z^{H,q}_{t_{j-1}}) - {\rm proj}_{\hat{E}_{j-1}}(D(Z^{H,q}_{t_j}-Z^{H,q}_{t_{j-1}})) \big\|_{\HHH}^2
\;\ge\; \|DZ^{H,q}_1\|_{\HHH_1}^2,
\]
which completes the proof.
\end{proof}

\section{Proof of Theorem \ref{thm1}}\label{sec:proof}

We are now in position to prove our main result.  
Combining the determinant identity for the Malliavin matrix (Lemma \ref{Lemma:detGamma}) with the strong local nondeterminism property for Hermite processes (Theorem \ref{Thm:MLNDDiscrete}), we can apply the Bouleau--Hirsch criterion to conclude.  

\begin{proof} 
Let 
$
\mathbf{Z}=(Z^{H,q}_{t_1},\ldots, Z^{H,q}_{t_n}),
$
and denote by $\Gamma_{\mathbf{Z}}$ its Malliavin matrix.  
By the Bouleau--Hirsch criterion (Theorem \ref{Thm:2.1.2}), it is enough to prove that
\[
\mathbb{P}\!\left( \det (\Gamma_{\mathbf{Z}})>0\right)=1.
\]

By Lemma \ref{Lemma:detGamma}, we can factorize the determinant as
\begin{equation}\label{eq:detfactor}
\det(\Gamma_{\mathbf{Z}}) 
= \|DZ^{H,q}_{t_1}\|_{\HHH}^2
   \prod_{j=2}^n \big\|DZ^{H,q}_{t_j} - \proj_{E_{j-1}}(DZ^{H,q}_{t_j})\big\|_{\HHH}^2,
\end{equation}
where $E_{j-1}=\mathrm{span}\{DZ^{H,q}_{t_1},\ldots,DZ^{H,q}_{t_{j-1}}\}$.

Now, by the self-similarity of the Malliavin derivative of $Z^{H,q}$ (Proposition \ref{Prop:Zt9999}) and Lemma \ref{lem:positive}, the first factor is almost surely positive.
For each $j=2,\ldots,n$, Theorem \ref{Thm:MLNDDiscrete} yields
\[
\big\|DZ^{H,q}_{t_j} - {\rm proj}_{E_{j-1}}(DZ^{H,q}_{t_j})\big\|_{\HHH}^2
\;\overset{\text{law}}{\ge}\; (t_j-t_{j-1})^{2H}\,\|DZ^{H,q}_1\|_{\HHH_1}^2,
\]
with $\HHH_1=L^2([0,1])$.  
Since $\|DZ^{H,q}_1\|_{\HHH_1}^2>0$ almost surely (again by Lemma \ref{lem:positive}), it follows that each factor in \eqref{eq:detfactor} is almost surely positive.

Therefore $\det(\Gamma_{\mathbf{Z}})>0$ almost surely, and by Theorem \ref{Thm:2.1.2} the law of $$(Z^{H,q}_{t_1},\ldots,Z^{H,q}_{t_n})$$ admits a density with respect to Lebesgue measure.  
This completes the proof of Theorem \ref{thm1}.
\end{proof}

\section*{Acknowledgements}

I.N. gratefully acknowledges support from the Luxembourg National Research Fund (Grant O22/17372844/FraMStA).  C. T.  acknowledges support from   the  ANR project SDAIM 22-CE40-0015,  MATHAMSUD grant 24-MATH-04 SDE-EXPLORE and  by the Ministry of Research, Innovation and Digitalization (Romania), grant CF-194-PNRR-III-C9-2023.
Part of this work was carried out while Y. N. was at the University of Luxembourg under a grant funding from the European Union's Horizon 2020 research and innovation programme N° 811017.

\end{document}